\documentclass{article}

\usepackage{graphicx} 
\usepackage{amsmath,amssymb, amsthm}
\usepackage{thmtools}

\usepackage[dvipsnames]{xcolor}
\usepackage{xspace}
\usepackage[most]{tcolorbox}
\usepackage{thmtools}
\usepackage{mathtools}
\usepackage{enumitem}
\usepackage{thm-restate}
\newtheorem{theorem}{Theorem}[section]

\newtheorem{lemma}[theorem]{Lemma}

\newtheorem{proposition}[theorem]{Proposition}

\newtheorem{observation}[theorem]{Observation}
\declaretheorem[name=Claim,style=remark]{claim}

\newcommand*\samethanks[1][\value{footnote}]{\footnotemark[#1]}
\newcommand{\dist}{\text{dist}}
\newcommand{\supp}{\text{supp}}

\newcommand{\mydiamond}{\rotatebox[origin=c]{45}{$\vcenter{\hbox{$\Box$}}$}}
\newenvironment{subproof}[1]{\par\noindent \textit{Proof of #1.}\ }{\hfill \mydiamond \par\vspace{11pt}}
\usepackage[english]{babel}

\usepackage[letterpaper,top=2cm,bottom=2cm,left=3cm,right=3cm,marginparwidth=1.75cm]{geometry}

\usepackage{amsmath}
\usepackage{graphicx}
\usepackage[colorlinks=true, allcolors=blue]{hyperref}
\usepackage{authblk}
\usepackage[noabbrev,capitalise,nameinlink]{cleveref}

\crefname{claim}{Claim}{Claims}

\begin{document}

\title{Asymmetric Homomorphism Thresholds for Graphs of Large Odd Girth}
\author[1]{Romain Bourneuf\samethanks}
\author[2]{Raphael Steiner\samethanks}
\author[3]{Stéphan Thomassé\samethanks}
\author[4]{Yuval Wigderson\samethanks}

\affil[1]{Univ. Bordeaux, CNRS, Bordeaux INP, LaBRI, UMR 5800, F-33400 Talence, France.}
\affil[2]{Department of Mathematics, ETH Zürich, 8092 Z\"urich, Switzerland.}
\affil[3]{ENS de Lyon, Université Claude Bernard Lyon 1, CNRS, Inria, LIP UMR 5668, Lyon, France.}
\affil[4]{Institute of Science and Technology Austria, 3400 Klosterneuburg, Austria.}

\maketitle

\begin{abstract}
We determine the asymmetric homomorphism threshold from graphs of odd girth at least $7$ to triangle-free graphs, showing that $\delta_{\mathrm{hom}}(\{C_3,C_5\},\{C_3\})=\frac{1}{9}$.
Equivalently, for every $\varepsilon>0$, every $n$-vertex graph of odd girth at least $7$ and minimum degree at least $(1/9+\varepsilon)n$ admits a homomorphism to a triangle-free graph of size bounded by a function of $\varepsilon$, while there exist graphs of odd girth at least $7$ and minimum degree at least $(1/9-\varepsilon)n$ for which no such bounded-size triangle-free homomorphic image exists.

More generally, for every $t\geq 3$, we prove $\delta_{\mathrm{hom}}(\{C_3,C_5\},\{K_t\})=\frac{1}{3t}$.
In particular, for every proper monotone class $\mathcal C$ of graphs, the threshold for graphs of odd girth at least $7$ to admit a homomorphism to a bounded-size graph in $\mathcal C$ is positive.
We further extend this phenomenon to arbitrary odd girth: for every $k\geq 2$ and every proper monotone subclass $\mathcal C$ of the class of graphs of odd girth at least $2k-1$, the threshold for graphs of odd girth at least $2k+3$ to admit a homomorphism to a bounded-size graph in $\mathcal C$ is positive.

These results disprove conjectures of Gishboliner, Hurley and Wigderson and exhibit a sharp contrast with the corresponding zero chromatic-threshold results. Our lower-bound constructions are based on high-dimensional Borsuk graphs, while the matching upper bounds use regularity arguments to recover the structure underlying these constructions.

\end{abstract}

\section{Introduction}
A central theme in extremal combinatorics is to study how the imposition of local constraints on a discrete object (e.g.\ forbidding a constant-sized subgraph) can yield consequences on the global structure. Many results in this direction state that if we impose a local constraint on a graph, then the global structure must become increasingly rigid as we insist on having more and more edges. For example, consider the constraint of having no triangles. For graphs with very few edges, this imposes essentially no constraint on the global structure. However, as the number of edges increases, the structure becomes more rigid; for example, an elegant result of Andr\'asfai \cite{A62} states that if the minimum degree of a triangle-free graph $G$ is more than $\frac 25 |V(G)|$, then $G$ must be bipartite, i.e.\ have a very rigid global structure. Moreover, once the number of edges increases beyond $\frac 14 |V(G)|^2$, the structure becomes even more rigid---such a graph simply does not exist \cite{mantel1907vraagstuk}. We remark that several recent breakthroughs in Ramsey theory and graph theory rely on the discovery that there exist highly unstructured triangle-free (or, more generally, $K_s$-free) graphs with substantially more edges than naive probabilistic reasoning would suggest \cite{2605.28793,2505.13371,2510.19718,2512.20392,2607.16118}.

In order to discuss this structural rigidity further, we introduce the following definition.
Given a family $\mathcal{F}$ of graphs, its \emph{chromatic threshold} $\delta_{\chi}(\mathcal{F})$ is the infimum of all $c \in [0, 1]$ for which there exists $C \in \mathbb{N}$ such that every $n$-vertex $\mathcal{F}$-free (i.e.\ which does not contain any $F \in \mathcal{F}$ as a (not necessarily induced) subgraph) graph with minimum degree at least $cn$ has chromatic number at most $C$. Thus, for example, the result of Andr\'asfai mentioned above shows that $\delta_\chi(C_3)\leq 2/5$.
A construction of Hajnal (see \cite{ES73}) shows that $\delta_{\chi}(C_3) \geq 1/3$.
Nearly 30 years later, Thomassen \cite{T02} proved a matching upper bound, thereby establishing  that $\delta_{\chi}(C_3) = 1/3$.
Brandt and Thomassé \cite{BT04} strengthened Thomassen's result by showing that $n$-vertex triangle-free graphs with minimum degree greater than $n/3$ have chromatic number at most $4$.
Thomassen \cite{T07} then proved that $\delta_{\chi}(C_5)=0$, which in particular implies $\delta_\chi(\{C_3,C_5\})=0$.
Said differently, the chromatic threshold of graphs with odd girth\footnote{The \emph{odd girth} of a graph $G$, denoted $og(G)$, is the length of a shortest odd cycle in $G$ (or $\infty$ if $G$ contains no odd cycle).} at least $7$ is $0$.
To see it, let $G$ be a graph with odd girth at least $7$ and minimum degree $cn$ for some $c > 0$. 
Let $\{v_1, \ldots, v_k\} \subseteq V(G)$ be maximal such that $N(v_i) \cap N(v_j) = \emptyset$ for all $i \neq j$. 
Since $G$ has minimum degree at least $cn$, it follows that $k \leq 1/c$. 
Since the set $\{v_1, \ldots, v_k\}$ is maximal, every $v \in V(G)$ is at distance at most two from some $v_i$. 
However, since $G$ contains no $C_3$ and no $C_5$, $N(v_i)$ and $N^2(v_i) \cup \{v_i\}$ are independent sets for every $1\le i\le k$, which together cover the full vertex set of $G$. 
Thus, $G$ has chromatic number at most $2k\le 2/c$.

While the chromatic threshold captures one way of formalizing the structural rigidity of dense $\mathcal F$-free graphs, we would often like to understand their structure more precisely. Indeed, while having bounded chromatic number is a natural way of insisting that a graph be ``highly structured'', it is natural to expect that dense $\mathcal F$-free graphs should have some additional structure reflecting their $\mathcal F$-freeness. For instance, in the result of Andr\'asfai discussed above, we see that a sufficiently dense triangle-free graph is bipartite, and of course all bipartite graphs are triangle-free; hence, the structural rigidity also provides a \emph{witness} for the $\mathcal F$-freeness in this case.

To formalize this idea, we recall the following definition.
The \emph{homomorphism threshold} $\delta_{\hom}(\mathcal{F})$ of $\mathcal{F}$ is the infimum of all $c \in [0, 1]$ such that there exists a finite $\mathcal{F}$-free graph $H_c$ such that every $n$-vertex $\mathcal{F}$-free graph with minimum degree at least $cn$ admits a homomorphism to $H_c$.
Equivalently\footnote{When every graph in $\mathcal{F}$ is connected, as is the case in the paper.}, $\delta_{\hom}(\mathcal{F})$ is the infimum of all $c$ such that every  $n$-vertex $\mathcal{F}$-free graph with minimum degree at least $cn$ admits a homomorphism to an $\mathcal{F}$-free graph of bounded size, where the bound depends only on $c$.
It is immediate from the definition that $\delta_{\chi}(\mathcal{F}) \leq \delta_{\hom}(\mathcal{F})$.
These two values sometimes coincide: Łuczak \cite{L06} proved that $\delta_{\hom}(C_3) = 1/3$.
However, they are not always equal, for instance $\delta_{\hom}(\{C_3, C_5\}) = 1/5$, as proved by Letzter and Snyder \cite{LS18}.
The lower bound follows from a construction of Andr\'{a}sfai \cite{A62,A64}, see also \cite{E57}: Consider the graph on the unit circle $\mathbb S^1$ where two vertices $u,v$ are adjacent whenever the angle between them is greater than $4\pi/5$. This graph is $\{C_3, C_5\}$-free and contains $n$-vertex induced subgraphs with minimum degree arbitrarily close to $n/5$ with no homomorphism to any fixed $\{C_3, C_5\}$-free graph.
We remark that there have recently been a number of exciting new results on homomorphism thresholds, e.g.\ \cite{HRS26,gishboliner2026homomorphismvcdimensionthresholdsspectra,HLZH26,Sankar,WWX26}.

A natural relaxation of the homomorphism threshold is to require the existence of a homomorphism to a graph $H$ that is not $\mathcal{F}$-free, but $\mathcal{F}'$-free, for some $\mathcal{F}' \neq \mathcal{F}$.
Following the notation of Gishboliner, Hurley and Wigderson \cite{GHW26}, given two families $\mathcal{F}_1, \mathcal{F}_2$ of graphs, their \emph{asymmetric homomorphism threshold} $\delta_{\hom}(\mathcal{F}_1, \mathcal{F}_2)$ is the infimum of all $c \in [0, 1]$ such that there exists a finite $\mathcal{F}_2$-free graph $H_c$ such that every $n$-vertex $\mathcal{F}_1$-free graph with minimum degree at least $cn$ admits a homomorphism to $H_c$.
In \cite{GHW26}, they show that $\delta_{\hom}(\{C_3, C_5, C_7\}, \{C_3\}) = 0$, which contrasts with the result of Ebsen and Schacht \cite{ES20} that $\delta_{\hom}(\{C_3, C_5, C_7\}) = 1/7$.
This comparison isolates what the positive symmetric threshold measures: the obstruction is not to bounded homomorphic compression itself, but to compression that preserves the full odd-girth condition. Indeed, every positive minimum-degree ratio forces a bounded triangle-free image, whereas below $1/7$ one cannot in general choose this image to avoid both $C_5$ and $C_7$. This naturally raises the question of how much of the forbidden-cycle structure of the source can always be retained in a bounded image.
In this direction, Gishboliner, Hurley and Wigderson conjectured that $\delta_{\hom}(\{C_3, C_5\}, \{C_3\}) = 0$, which would strengthen Thomassen's result that $\delta_{\chi}(\{C_3, C_5\}) = 0$ and pinpoint the complexity of the homomorphism threshold for $\{C_3, C_5\}$-free graphs. Our first main result is to disprove this conjecture, by precisely determining $\delta_{\hom}(\{C_3, C_5\}, \{C_3\})$ as an absolute positive constant.

\begin{restatable}{theorem}{thresholdtriangle}\label{thm:threshold-triangle}
    $\delta_{\hom}(\{C_3, C_5\}, \{C_3\}) = \frac{1}{9}$.
\end{restatable}

Observe that $\delta_{\chi}(\{C_3, C_5\}) = 0$ is equivalent to $\delta_{\hom}(\{C_3, C_5\}, \emptyset) = 0$. 
It is then natural to wonder whether this statement can be strengthened at all, namely  whether there exists a proper monotone\footnote{A class of graphs is \emph{monotone} if it is closed under taking subgraphs. When an ambient monotone class $\mathcal G$ is specified, a monotone subclass $\mathcal C\subseteq\mathcal G$ is \emph{proper} if $\mathcal C\neq\mathcal G$.} class $\mathcal{C}$ of graphs such that the threshold for $\{C_3, C_5\}$-free graphs to have a homomorphism to a bounded-size graph in $\mathcal{C}$ is $0$. 
We show that this is not the case.

\begin{theorem}\label{thm:threshold-Kt}
    For every $t \geq 3$, we have $\delta_{\hom}(\{C_3, C_5\}, \{K_t\})  = \frac{1}{3t}$.
    Therefore, for every graph $H$, we have $\delta_{\hom}(\{C_3, C_5\}, \{H\}) > 0$.
\end{theorem}

\cref{thm:threshold-Kt} showcases a surprising contrast between the threshold for the existence of homomorphisms to bounded-size graphs and the threshold for the existence of homomorphisms to bounded-size graphs from a proper monotone class for $\{C_3, C_5\}$-free graphs.
Indeed, for $\{C_3,C_5\}$-free graphs, there is a sharp qualitative change between merely forcing a bounded homomorphic image and forcing such an image to satisfy any fixed nontrivial monotone restriction.

The result of Thomassen was extended by Gishboliner, Hurley and Wigderson \cite{GHW26}, who showed that for every $k \geq 2$, $\delta_{\hom}(\{C_3, C_5, \ldots, C_{2k+1}\}, \{C_3, \ldots, C_{2k-3}\}) = 0$.
They conjectured that this result could be strengthened to $\delta_{\hom}(\{C_3, C_5, \ldots, C_{2k+1}\}, \{C_3, \ldots, C_{2k-3}, C_{2k-1}\}) = 0$.
Again, it is natural to relax their statement to ask whether there exists a proper monotone subclass $\mathcal{C}$ of the class of graphs of odd girth at least $2k-1$ such that the threshold for graphs of odd girth at least $2k+3$ to have a homomorphism to a bounded-size graph in $\mathcal{C}$ is $0$.
We also refute this more general statement.

\begin{restatable}{theorem}{thresholdoddgirth}\label{thm:threshold-odd-girth}
    For every $k \geq 2$, for every graph $H$ of odd girth at least $2k-1$, we have \[\delta_{\hom}(\{C_3, C_5, \ldots, C_{2k+1}\}, \{C_3, \ldots, C_{2k-3}, H\}) > 0.\]
\end{restatable}

In \cref{thm:threshold-Kt}, we are able to pinpoint the actual threshold for cliques. 
We then derive the result using that the family of cliques is \emph{universal}: every graph $H$ is a subgraph of some clique.
For $k \geq 2$, the class of \emph{$(k-2)$-fold Mycielskians} is a class of graphs of odd girth at least $2k-1$ that is universal for the class of graphs of odd girth at least $2k-1$.
To prove \cref{thm:threshold-odd-girth}, for every $k \geq 2$ and every $(k-2)$-fold Mycielskian $\Gamma$, we in fact establish that $\delta_{\hom}(\{C_3, C_5, \ldots, C_{2k+1}\}, \{C_3, \ldots, C_{2k-3}, \Gamma\})  = \frac{1}{3|V(\Gamma)|}$, and then immediately derive \cref{thm:threshold-odd-girth} from there.

\subsection{Proof sketch of the lower bound}
In our opinion, the most surprising finding in this paper is the lower bound in \cref{thm:threshold-triangle}, and specifically the counterexample to the conjecture of \cite{GHW26} that $\delta_{\hom}(\{C_3,C_5\},\{C_3\})=0$. As such, we end this introduction by sketching a proof of the lower bound $\delta_{\hom}(\{C_3,C_5\},\{C_3\})\geq \frac 19$; the proof of the more general \cref{thm:threshold-odd-girth} follows the same approach, albeit with more technical complications to obtain the general result. We remark that the proof of the matching upper bound $\delta_{\hom}(\{C_3,C_5\},\{C_3\})\leq \frac 19$ is also inspired by the structure of our lower-bound construction, and the proof of the upper bound essentially ``uncovers'' this structure. Thus, while this is by no means the unique construction witnessing $\delta_{\hom}(\{C_3,C_5\},\{C_3\})\geq \frac{1}{9}$, it is the case that all such constructions must be of a similar form.

The key object in the construction is the \emph{Borsuk graph} $Bor(d,\theta)$. This is an infinite graph, whose vertex set is the $d$-dimensional unit sphere $\mathbb S^d$, and in which two vertices $u,v$ are adjacent if and only if $\dist(u,v)\geq \pi-\theta$, where $\dist$ denotes the angular distance between $u$ and $v$. The main property we require of the Borsuk graph (for now) is that $\chi(Bor(d,\theta))\geq d+2$ for any $d,\theta$; this is a simple but important consequence of the Borsuk--Ulam theorem \cite{B33} (and the bound $d+2$ is optimal for all sufficiently small $\theta$). Although $Bor(d,\theta)$ is an infinite graph, one can approximate it by finite graphs by taking the induced subgraph on a sufficiently fine $\varepsilon$-net; in this high-level proof overview we do not dwell on this point.

Before presenting our construction, let us first present the construction of Hajnal showing that $\delta_\chi(C_3)\geq \frac 13$. More precisely, we will discuss a variant of Hajnal's construction due to \L uczak and Thomass\'e \cite{1007.1670}, which is based on Borsuk graphs. We define a graph $G_0$ on vertex parts $X,Y,Z$
as follows. First, $X$ is a copy of $Bor(d,\theta)$ for some large $d$ and small $\theta$, to be chosen later, while $Y$ and $Z$ are independent sets. Between $Y$ and $Z$ we put a complete bipartite graph, while between $X$ and $Y$ we put a ``hemisphere graph'': we identify the vertices of $Y$ with $\mathbb S^d$ as well, and we join $x \in X$ to $y \in Y$ if and only if $\dist(x,y) \leq \frac{\pi}2-\theta$. That is, for any $y \in Y$, its neighborhood in $X$ is essentially the hemisphere centered at $y$; more precisely, it is the spherical cap centered at $y$ which is $\theta$-far from being a full hemisphere. 

Since $Bor(d,\theta)$ is an induced subgraph of $G_0$, we see that $\chi(G_0)\geq d+2$. Hence, by picking $d$ appropriately, we can ensure that $\chi(G_0)$ is arbitrarily large. Moreover, it is not hard to check that $G_0$ is triangle-free for sufficiently small $\theta$: the key point here is that if $x,x' \in X$ are both adjacent to some $y \in Y$, then they both lie in the spherical cap around $y$, and hence their distance must be at most $\pi-2\theta$ by the triangle inequality, and thus they are non-adjacent in $Bor(d,\theta)\cong G_0[X]$. Note too that, as $\theta\to 0$, the edge density between $X$ and $Y$ tends to $\frac 12$, since the spherical cap approaches a hemisphere and hence picks up almost half the mass of the sphere. The final step is now to duplicate some vertices of $G_0$ in order to obtain the desired minimum degree. Namely, we duplicate vertices appropriately to ensure that roughly one-third of the vertices lie in $Z$ and roughly two-thirds of the vertices lie in $Y$ (and $X$ has a negligible fraction of the vertices). In this way, every vertex is adjacent to almost one-third of the other vertices, completing the proof that $\delta_\chi(C_3)\geq \frac 13$. 

Our construction for the lower bound on $\delta_{\hom}(\{C_3,C_5\},\{C_3\})$ can be viewed as ``gluing together'' three copies of $G_0$. Namely, let us define a graph $G$ on vertex sets $X_1,X_2,X_3,Y_1,Y_2,Y_3,Z_1,Z_2,Z_3$. As before, we make each $(Y_i,Z_i)$ complete bipartite, and each $(X_i,Y_i)$ a ``hemisphere graph''. The main twist is that now each $X_i$ is also an independent set, rather than a Borsuk graph. Instead, we place the tripartite cover of the Borsuk graph on the vertices $X_1 \cup X_2 \cup X_3$. That is, for all $i \neq j$ and all $x \in X_i, x' \in X_j$, we make $xx'$ an edge of $G$ if and only if $\dist(x,x')\geq \pi-\theta$. By duplicating vertices as above, we can ensure that the minimum degree of $G$ is arbitrarily close to $\frac 19 |V(G)|$ (just because the minimum degree in each $X_i \cup Y_i \cup Z_i$ is one-third, and there are three such parts). It is also not hard to check that $G$ is $\{C_3,C_5\}$-free; the key point now is that for any $y \in Y_1$ (say), any two of its neighbors in $X_1$ cannot be joined by a path of length $3$, as any two vertices in $X_1$ at distance $3$ must still be nearly antipodal, and no two elements of a spherical cap are nearly antipodal.

So it remains to prove that $G$ has no bounded-size triangle-free homomorphic image. For contradiction, suppose that $\varphi:G \to H$ is a homomorphism, where $H$ is some triangle-free graph. Let $x,x' \in X_1$ with $\varphi(x)=\varphi(x')$. We note that $x,x'$ cannot be joined by a path of length $3$ in $G$, for if they were, the image of this path under $\varphi$ would yield a triangle in $H$. Hence, if $\varphi(x)=\varphi(x')$, then $\dist(x,x')\leq \pi-3\theta$. That is, $\varphi$ defines a proper coloring of $Bor(d,3\theta)$ with $|H|$ colors, and hence $|H|\geq \chi(Bor(d,3\theta))\geq d+2$. Thus, by picking $d$ appropriately, we can rule out triangle-free homomorphic images of $G$ of any fixed size, as claimed.

\paragraph*{Acknowledgments.} Part of this work was carried out while the first and third authors were visiting the Institute for Theoretical Studies at ETH Z\"urich. We thank them for providing a wonderful and productive working environment. 

\paragraph*{AI Disclosure.} We used assistance from ChatGPT 5.5 Pro, which helped us discover 
the proof for the upper bound in Theorem~\ref{thm:threshold-triangle}. All remaining proofs and results are due to the authors. Moreover, this paper, including all its proofs, is fully written by the (human) authors.

\section{Non-existence of homomorphisms}

In this section, we prove the lower bounds of \cref{thm:threshold-Kt} and \cref{thm:threshold-odd-girth}.
For this, we need some preparation.

\subsection{Generalized Mycielskians}

For an integer $n$, we denote by $[n]$ the set $\{1, \ldots, n\}$.
Given an integer $k \geq 0$ and a graph $\Gamma$, its \emph{$k$-fold Mycielskian} is the graph $M_k(\Gamma)$ defined as follows: \begin{enumerate}
    \item $V(M_k(\Gamma)) = (V(\Gamma) \times [k+1]) \cup \{r\}$. That is, the vertices of $M_k(\Gamma)$ are of the form $(v, i)$ for $v \in V(\Gamma)$ and $i \in [k+1]$, as well as a single fresh vertex $r$.
    \item The induced subgraph on $V(\Gamma) \times \{k+1\}$ is a copy of $\Gamma$ and $V(\Gamma) \times \{i\}$ is an independent set for all $i \in [k]$.
    \item $r$ is adjacent to all vertices in $V(\Gamma) \times \{1\}$ and no other vertices.
    \item For all $i \in [k]$, each vertex $(v, i)$ is adjacent to all other vertices $(w, i+1)$ such that $vw \in E(\Gamma)$ and to no other vertices of $V(\Gamma) \times \{i+1\}$.
\end{enumerate}

The sequence of $k$-fold Mycielskians is the sequence $(\Gamma_k^t)_{t \geq 2}$ where $\Gamma_k^2 = K_2$ and $\Gamma_k^{t+1} = M_k(\Gamma_k^t)$ for every $t \geq 2$.

We now review some basic properties of $(\Gamma_k^t)_{t \geq 2}$.
The first one was proved by Stiebitz \cite{S85} when he introduced generalized Mycielskians.
To the best of our knowledge, the second was first proved by Cropper, Gy\'arf\'as and Lehel \cite{CGL06} for Mycielskians. For completeness, we include a proof of the result for generalized Mycielskians.

\begin{lemma}[\cite{S85}]\label{lem:og-Myc}
    Let $k \geq 0$ and let $G$ be a graph.
    If $G$ has odd girth $2k+3$, so does $M_k(G)$.
    In particular, for every $t \geq 3$, $\Gamma_k^t$ has odd girth $2k+3$.
\end{lemma}

\begin{lemma}\label{lem:Myc-universal}
    Let $k \geq 0$ and let $G$ be a graph of odd girth greater than $2k+1$.
    Then, there exists $t \geq 2$ such that $G$ is a subgraph of $\Gamma_k^t$.
\end{lemma}

\begin{proof}
    We prove by induction on the number of vertices of $G$ that $G$ is a subgraph of $\Gamma_k^{|V(G)|}$.
    The statement is trivial if $|V(G)|=2$.
    Suppose now that the statement holds for all graphs on some number $n$ of vertices, and let a graph $G$ with $|V(G)| = n+1$ and odd girth greater than $2k+1$.
    Let us pick any $v \in V(G)$.
    Then $G-v$ is a subgraph of $\Gamma_k^n$ by the induction hypothesis so there exists an embedding $\psi : V(G-v) \to V(\Gamma_k^n)$ that witnesses it.
    We define an embedding $\phi : V(G) \to V(\Gamma_k^{n+1})$ as follows. \begin{itemize}
        \item $\phi(v) = r$.
        \item For every $u\in V(G)\setminus \{v\}$, if $\dist_G(v, u) \leq k$, then define $\phi(u) := (\psi(u), \dist_G(v, u))$.
        \item Otherwise, define $\phi(u) := (\psi(u), k+1)$.
    \end{itemize}
    Since $G$ has odd girth greater than $2k+1$, it follows that for every $i \in [k]$, the set of vertices at distance $i$ from $v$ is an independent set in $G$. It is then straightforward to verify that $\phi$ witnesses that $G$ is a subgraph
    of $\Gamma_k^{n+1}$.
\end{proof}
The next lemma is a simple observation about the existence of walks of prescribed length between different vertices of generalized Mycielskians.
\begin{lemma}\label{lem:exists-walk}
    Let $k \geq 1$ and let $\Gamma$ be a graph.
    Let $u, v \in V(\Gamma)$ such that there is a walk of length $\ell$ between $u$ and $v$ in $\Gamma$.
    Let $i \leq j \in [k+1]$.
    There exists a walk of length $\ell$ in $M_k(\Gamma)$ between $(u, i)$ and $(v, j)$ if one of the following holds: \begin{itemize}
        \item $\ell$ and $j-i$ have the same parity and $\ell \geq j-i$, or
        \item $\ell \geq 2k+2-(i+j)$.
    \end{itemize}
\end{lemma}

\begin{proof}
    Let $u = u_0 u_1 \ldots u_{\ell} = v$ be a walk of length $\ell$ between $u$ and $v$ in $\Gamma$.
    Suppose first that $\ell$ and $j-i$ have the same parity and $\ell \geq j-i$.
    Since $k \geq 1$, there exists $j' \in \{j-1, j+1\} \cap [k+1]$.
    Then, $(u, i)(u_1, i+1) \ldots(u_{j-i}, j)(u_{j-i+1}, j')(u_{j-i+2}, j)\ldots(u_{\ell-1}, j')(v, j)$ is a walk of length $\ell$ in $M_k(\Gamma)$ between $(u, i)$ and $(v, j)$. To see that this walk is well-defined, we use that $\ell \geq j-i$, and that $\ell$ and $j-i$ have the same parity (so that after the $\ell - (j-i)$ last steps we are indeed back at $(v, j)$ and not at $(v, j')$).

    Suppose now that $\ell \geq 2k+2-(i+j)$.
    Then, $(u, i)(u_1, i+1)\ldots(u_{k+1-i}, k+1)(u_{k+2-i}, k+1)\ldots(u_{j+\ell-k-1}, k+1)(u_{j+\ell-k}, k)\ldots(u_{\ell-1}, j+1)(v, j)$ is a walk of length $\ell$ in $M_k(\Gamma)$ between $(u, i)$ and $(v, j)$. 
    To see that this walk is well-defined, we need $k+1-i \leq j+\ell - k - 1$, i.e. $\ell \geq 2k+2-(i+j)$.
\end{proof}
Finally, we need the following lemma, which shows that generalized Mycielskians are ``rich'' in closed walks of length $2k+3$.
\begin{lemma}\label{lem:exists-closed-walk-Myc}
    Let $k \geq 0$ and let $\Gamma$ be a graph such that any two vertices of $\Gamma$ belong to a closed walk of length $2k+3$.
    Then any two vertices of $M_k(\Gamma)$ belong to a closed walk of length $2k+3$.
\end{lemma}

\begin{proof}
    Since any two vertices of $\Gamma$ belong to a closed walk of length $2k+3$, $\Gamma$ must contain an edge.
    For $k = 0$, observe that $M_k(\Gamma)$ is the graph obtained from $\Gamma$ by adding a universal vertex. It is then straightforward to check that any two vertices of $M_k(\Gamma)$ belong to a closed walk of length $3$.

    Suppose now that $k \geq 1$. 
    Recall that $V(M_k(\Gamma)) = (V(\Gamma) \times [k+1]) \cup \{r\}$.
    First, we prove that for every $(v, i) \in V(\Gamma) \times [k+1]$, $(v, i)$ and $r$ belong to a closed walk of length $2k+3$ in $M_{k}(\Gamma)$.
    By assumption, $v$ belongs to a closed walk of length $2k+3$ in $\Gamma$, so $v$ has a neighbor $u$ in $\Gamma$.
    Then, $(u, k+1)(v, k)(u, k-1)\ldots r\ldots(v, k-1)(u, k)(v, k+1)(u, k+1)$ is a closed walk of length $2k+3$ in $M_k(\Gamma)$ which visits $(v, i)$ and $r$.
    Note that this also shows that there is a walk of length $i$ and a walk of length $2k+3-i$ between $(v, i)$ and $r$.
    It also proves that for every $v \in V(\Gamma)$ and $i \neq j \in [k+1]$, $(v, i)$ and $(v, j)$ belong to a closed walk of length $2k+3$.
    
    Finally, let $u \neq v \in V(\Gamma)$ and $i, j \in [k+1]$.
    By symmetry, we can assume that $i \leq j$.
    Suppose first that there is a walk of length $j-i$ between $u$ and $v$ in $\Gamma$.
    By \cref{lem:exists-walk}, there is a walk of length $j-i$ in $M_k(\Gamma)$ between $(u, i)$ and $(v, j)$.
    Furthermore, there is a walk of length $i$ between $(u, i)$ and $r$, and a walk of length $2k+3-j$ between $(v, j)$ and $r$. 
    The concatenation of these three walks is a closed walk of length $2k+3$ that visits $(u, i)$ and $(v, j)$.
    Suppose now that there is a walk of length $2k+3-i-j$ between $u$ and $v$ in $\Gamma$.
    By \cref{lem:exists-walk}, there is a walk of length $2k+3-i-j$ in $M_k(\Gamma)$ between $(u, i)$ and $(v, j)$.
    Furthermore, there is a walk of length $i$ between $(u, i)$ and $r$, and a walk of length $j$ between $(v, j)$ and $r$. 
    The concatenation of these three walks is a closed walk of length $2k+3$ that visits $(u, i)$ and $(v, j)$.
    Finally, suppose that we are not in one of those two cases.
    By assumption, there is a closed walk $W$ of length $2k+3$ in $\Gamma$ that goes through $u$ and $v$.
    Write $W = W_1W_2$ such that $W_1$ and $W_2$ are both walks with endpoints $u$ and $v$.
    By symmetry, we can assume that the length of $W_1$, call it $d$, has the same parity as $j-i$. Note that $W_2$ has length $2k+3-d$.
    Since there is no walk of length $j-i$ in $\Gamma$ between $u$ and $v$ in $\Gamma$ and $d$ and $j-i$ have the same parity, we have $d > j-i$.
    By \cref{lem:exists-walk}, there is a walk of length $d$ between $(u, i)$ and $(v, j)$ in $M_k(\Gamma)$.
    Since there is no walk of length $2k+3-i-j$ between $u$ and $v$ in $\Gamma$ and $2k+3-d$ and $2k+3-i-j$ have the same parity, we have $2k+3-d > 2k+3-i-j$.
    By \cref{lem:exists-walk}, there is a walk of length $2k+3-d$ between $(u, i)$ and $(v, j)$ in $M_k(\Gamma)$.
    The concatenation of these two walks is a closed walk of length $2k+3$ that visits $(u, i)$ and $(v, j)$.
\end{proof}

Throughout the paper, we will implicitly use the following fact.

\begin{observation}\label{obs:odd-walk-cycle}
    Let $G$ be a graph that contains a closed odd walk of length at most $\ell$.
    Then $G$ contains an odd cycle of length at most $\ell$.
\end{observation}

As a corollary of this observation and \cref{lem:exists-closed-walk-Myc}, we immediately obtain the following.

\begin{lemma}\label{lem:Myc-2-vtces-odd-cycle}
    Let $k \geq 0$ and $t \geq 3$.
    Then any two vertices of $\Gamma_k^t$ belong to a cycle of length $2k+3$.
\end{lemma}

\begin{proof}
    We prove it by induction on $t$, with $k \geq 0$ fixed.
    For $t = 3$, observe that $\Gamma_k^t = M_k(K_2) = C_{2k+3}$ so the property holds.
    Let $t \geq 3$ such that the property holds.
    Since $\Gamma_k^{t+1} = M_k(\Gamma_k^t)$, \cref{lem:exists-closed-walk-Myc} implies that any two vertices of $\Gamma_k^{t+1}$ belong to a closed walk of length $2k+3$.
    Since $\Gamma_k^{t+1}$ has odd girth at least $2k+3$ by \cref{lem:og-Myc}, this closed walk must be a cycle, which proves the desired result.
\end{proof}

Finally, the following lemma places restrictions on homomorphic images of generalized Mycielskians: every homomorphic image of $\Gamma_k^t$ with high odd girth must actually contain $\Gamma_k^t$ as an induced subgraph.

\begin{lemma}\label{lem:Myc-core}
    Let $k \geq 0$ and $t \geq 3$ and let $J$ be a graph of odd girth at least $2k+3$ such that there exists a homomorphism from $\Gamma_k^t$ to $J$.
    Then $J$ contains $\Gamma_k^t$ as an induced subgraph.
\end{lemma}

\begin{proof}
    Let $\phi : V(\Gamma_k^t) \to V(J)$ be a homomorphism.
    First, we show that $\phi$ is injective.
    By contradiction, suppose not: there exist $x \neq y \in V(\Gamma_k^t)$ such that $\phi(x) = \phi(y)$.
    By \cref{lem:Myc-2-vtces-odd-cycle}, there exists a path $P$ between $x$ and $y$ in $\Gamma_k^t$ whose length is odd and smaller than $2k+3$.
    Since $\phi$ is a homomorphism, the image of $P$ under $\phi$ is a walk of the same length between $\phi(x)$ and $\phi(y)$.
    Thus, there is a closed walk in $J$ whose length is odd and smaller than $2k+3$.
    This contradicts the fact that $J$ has odd girth at least $2k+3$.
    Therefore, $J$ contains $\Gamma_k^t$ as a subgraph.
    By \cref{lem:Myc-2-vtces-odd-cycle}, any two vertices of $\Gamma_k^t$ belong to a cycle of length $2k+3$, so $\Gamma_k^t$ is edge-maximal with the property of having odd girth at least $2k+3$.
    Since $J$ has odd girth at least $2k+3$, this implies that the copy of $\Gamma_k^t$ in $J$ is induced, as desired.
\end{proof}

\subsection{High-dimensional spheres and Borsuk graphs}

We denote by $\mathbb{S}^d$ the unit sphere of $\mathbb R^{d+1}$. Given $x\in \mathbb{S}^d$, its \emph{antipodal} point is $-x$. 
We consider the distance $\dist : \mathbb{S}^d \times \mathbb{S}^d \to \mathbb{R}^+$ such that for all $u, v \in \mathbb{S}^d$, $\dist(u, v) \in [0, \pi]$ is the angular distance between $u$ and $v$. For instance, antipodal points are at distance $\pi$ while orthogonal points are at distance $\pi/2$. 

We recall the construction of the (infinite) \emph{$d$-dimensional $\theta$-Borsuk graph} $Bor(d,\theta)$ whose vertex set consists of the unit sphere $\mathbb{S}^d$ and where $u, v \in \mathbb{S}^d$ are adjacent if $\dist(u, v) \geq \pi-\theta$. 
It is well-known (cf.~\cite{EH67}) that the odd girth of $Bor(d,\theta)$ is arbitrarily large when $\theta>0$ is sufficiently small and that the fractional chromatic number\footnote{The fractional chromatic number $\chi_f(G)$ of a graph $G$ is the infimum of $\sum_{I\in\mathcal I(G)} w_I$ over all choices of nonnegative weights $(w_I)_{I\in\mathcal I(G)}$ on the independent sets of $G$ such that, for every vertex $v\in V(G)$, we have $\sum_{I\ni v} w_I\geq 1$.
}
of $Bor(d,\theta)$ tends to $2$ as $\theta$ approaches $0$. 

Arguments concerning the complexity of the $d$-dimensional $\theta$-Borsuk graph typically rely on the famous \emph{Borsuk--Ulam theorem} \cite{B33}, which can be equivalently cast as the fact that $Bor(d, \theta)$ has chromatic number at least $d+2$. 
See \cite{M03} for an overview of the applications of the Borsuk-Ulam theorem in combinatorics.

\subsection{A general construction}
The following proposition is a generalization of the construction sketched in the introduction. Rather than simply placing a tripartite Borsuk graph on the parts $X_1,X_2,X_3$ as we did there, we instead use a given graph $H$ as a ``template'' for how to place Borsuk graphs; in the example in the introduction, this template graph $H$ would be a triangle.

\begin{proposition}\label{prop:construction}
    For every integer $N \geq 1$, for every $\varepsilon > 0$, and for every graph $H$ which contains an odd cycle, there exists a graph $H'$ with the following properties: \begin{itemize}
        \item $og(H') \geq og(H) + 4$,
        \item $\delta(H') \geq \left(\frac{1}{3|V(H)|} - \varepsilon\right)|V(H')|$,
        \item If $H'$ has a homomorphism to a graph $J$ on at most $N$ vertices then $H$ also has a homomorphism to $J$.
    \end{itemize}
\end{proposition}

\begin{proof}
    Let $d \in \mathbb{N}$ be large enough so that for every $\theta > 0$, the Borsuk graph $Bor(d, \theta)$ has chromatic number greater than $N^{|V(H)|}$.
    Define $\zeta = \varepsilon|V(H)|$ and let $\eta>0$ be chosen small enough so that $\frac{1}{2+\eta}\cdot \frac{1}{3/2 + \eta} \geq \frac{1}{3} - \zeta$.
    Let $\theta \in (0, \pi)$ be small enough so that $\chi_f(Bor(d, \theta)) \leq 2 + \eta$ and let $\iota < \frac{\theta}{og(H)+2}$.
    By the \emph{de Bruijn--Erd\H{o}s theorem}, there exists a finite set $X \subseteq \mathbb{S}^d$ such that $\chi(Bor(d, \iota)[X]) = \chi(Bor(d, \iota)) > N^{|V(H)|}$.
    Let $F \coloneqq Bor(d, \theta)[X]$, and observe that $\chi_f(F) \leq \chi_f(Bor(d, \theta)) \leq 2 + \eta$.

    Let $\mathcal{I}$ be the set of all independent sets of $F$.
    Since $\chi_f(F) \leq 2 + \eta$, there exists a weight function $w : \mathcal{I} \to \mathbb{R}^+$ such that $\sum_{v \in I \in \mathcal{I}}w(I) \geq 1$ for every $v \in X$ and $w(\mathcal{I}) \coloneqq \sum_{I \in \mathcal{I}}w(I) \leq 2+\eta$.
    Since $\chi_f(F)$ is the solution of a (finite) linear program with integral coefficients, we may assume without loss of generality that $w : \mathcal{I} \to \mathbb{Q}^+$.
    Thus, after rescaling, there exists a function $w' : \mathcal{I} \to \mathbb{N}$ such that $\sum_{v \in I \in \mathcal{I}}w'(I) \geq \frac{1}{2+\eta} \cdot w'(\mathcal{I})$ for every $v \in X$.
    By considering the (bipartite) incidence graph between the vertices $X$ and the independent sets $\mathcal{I}$, and replacing each vertex corresponding to an independent set $I \in \mathcal{I}$ by $w'(I)$ copies of itself, it follows that there exists a set $Y$ and a bipartite graph $B$ with bipartition $(X, Y)$ such that $|N_B(v)| \geq \frac{1}{2+\eta} \cdot |Y|$ for every $v \in X$.
    Note that $Y$ can be chosen to have even size and arbitrarily large, by replacing each element of $Y$ by a large (but fixed) number of copies of itself.
    In particular, we can assume that $|Y|$ is even and that $|X| \leq \eta|Y|$.
    By adding a set $Z$ of size $|Y|/2$ that is complete to $Y$, we obtain a bipartite graph $G$ with bipartition $(X \cup Z, Y)$.
    Note that $|V(G)| = |X| + |Z| + |Y| = |X| + \frac{3}{2}|Y| \leq \left(\frac{3}{2} + \eta\right)|Y|$.
    Then, every $v \in X \cup Z$ satisfies $|N_{G}(v)| \geq \frac{1}{2+\eta} \cdot |Y| \geq \frac{1}{2 + \eta} \cdot \frac{1}{3/2 + \eta}|V(G)| \geq  \left(\frac{1}{3} - \zeta\right)|V(G)|$ by definition of $\eta$.
    Also, every $v \in Y$ satisfies $|N_{G}(v)| \geq |Z| = \frac{|Y|}{2} \geq \frac{1}{2} \cdot \frac{|V(G)|}{3/2 + \eta} \geq \frac{1}{2 + \eta} \cdot \frac{1}{3/2 + \eta}|V(G)| \geq \left(\frac{1}{3} - \zeta\right)|V(G)|$ again by definition of $\eta$.
    Thus, $\delta(G) \geq \left(\frac{1}{3} - \zeta\right)|V(G)|$.

    For every $h \in V(H)$, let $G_h$ be a copy of $G$ such that all copies $G_h$ are vertex-disjoint.
    For every $h \in V(H)$, denote the vertex set of $G_h$ by $X_h \cup Y_h \cup Z_h$, and for every $x \in X$ let $x_h$ denote the copy of $x$ in $X_h$.
    Then, let $H'$ be the graph obtained from the disjoint union of all $G_h$ by adding edges between the sets $X_h$ as in the tensor (or categorical) product $Bor(d, \iota) \times H$.
    Formally, connect $x_h$ and $x'_{h'}$ for $x, x' \in X$ and $h, h' \in V(H)$ if and only if $xx' \in E(Bor(d, \iota))$ (i.e. if $\dist(x, x') \geq \pi - \iota$) and $hh' \in E(H)$.

    We now show that $H'$ satisfies the desired properties.
    By contradiction, suppose that $og(H') < og(H) + 4$, so $og(H') \leq og(H) + 2$.
    Since we can always extend the length of a closed walk by any even number by moving forth and back along some edge, it follows that there exists a closed walk $W'$ of length \emph{exactly} $\ell \coloneqq og(H) + 2$ in $H'$.
    Let $H^+$ be the graph obtained from $H$ by adding for every $h \in V(H)$ a disjoint edge $e_h \coloneqq a_hb_h$ which we connect to $h$ via another edge $ha_h$. 
    By construction, the map $\phi : V(H') \to V(H^+)$ that maps each set $X_h$ to $h$, each set $Y_h$ to $a_h$ and each set $Z_h$ to $b_h$, is a homomorphism from $H'$ to $H^+$.
    Thus, $\phi$ maps $W'$ to a closed walk $W^+$ of length $\ell$ in $H^+$.
    Suppose first that $W^+$ only visits vertices of $H$ in $H^+$.
    Then, we can write $W' = x^0_{h_0}x^1_{h_1}\ldots x^{\ell-1}_{h_{\ell-1}}x^0_{h_0}$ where $h_i \in V(H)$ and $x_i \in X$ for every $i \in \{0, \ldots, \ell\}$.
    A simple proof by induction shows that $\dist(x^0, x^j) \leq j\iota$ for every even $j \in \{0, \ldots, \ell-1\}$.
    Since $\ell-1$ is even, we have $\dist(x^0, x^{\ell-1}) \leq (\ell-1)\iota$.
    However, since $x^{\ell-1}_{h_{\ell-1}}x^0_{h_0}$ is an edge of $H'$, we have $\dist(x^{0}, x^{\ell-1}) \geq \pi-\iota$, so $\pi-\iota \leq (\ell-1)\iota$ so $\iota \geq \frac{\pi}{\ell} = \frac{\pi}{og(H)+2}$, a contradiction since $\iota < \frac{\theta}{og(H)+2} < \frac{\pi}{og(H)+2}$ by definition.
    Thus, $W^+$ does not only visit vertices from $V(H)$ in $H^+$.
    Since $W^+$ has length $og(H)+2$ and every odd cycle of $H^+$ lives in $V(H)$, so has length at least $og(H)$, it follows that there exists $h \in V(H)$ so that $W^+$ visits $h$, then $a_h$, then $h$ again, and then forms a cycle of length $og(H)$ that only visits vertices of $H$ in $H^+$.
    Let $x^0_h$ and $x^2_h$ be the two vertices visited by $W'$ that belong to $X_h$.
    Then, $x^0_h$ and $x^2_h$ have a common neighbor in $Y_h$, and so there is an independent set in $F = Bor(d, \theta)[X]$ containing both $x^0$ and $x^2$. Hence, $x^0$ and $x^2$ are not adjacent in $F = Bor(d, \theta)[X]$, so $\dist(x^0, x^2) \leq \pi - \theta$.
    However, as before, a simple proof by induction shows that $\dist(x^0, x^2) \geq \pi - og(H)\iota$.
    This implies $\pi-og(H)\iota \leq \pi - \theta$, so $\iota \geq \frac{\theta}{og(H)}$, again a contradiction.
    Thus, $og(H') \geq og(H) + 4$.

    By definition, we have $|V(H')| = |V(H)| \cdot |V(G)|$.
    Let $v \in V(H')$ and let $h \in V(H)$ such that $v \in V(G_h)$.
    Then, \[d_{H'}(v) \geq d_{G_h}(v) \geq \left(\frac{1}{3} - \zeta\right)|V(G)| = \left(\frac{1}{3|V(H)|} - \frac{\zeta}{|V(H)|}\right)|V(H')| = \left(\frac{1}{3|V(H)|} - \varepsilon\right)|V(H')|.\]

    Let $J$ be a graph on at most $N$ vertices such that there exists a homomorphism $\phi : V(H') \to V(J)$ from $H'$ to $J$.
    Without loss of generality, let us assume that $V(J) = [N]$.
    Write $V(H) = \{h_1, \ldots, h_{|V(H)|}\}$.
    For each $x \in X$, let $c(x) \coloneqq (\phi(x_{h_1}), \ldots, \phi(x_{h_{|V(H)|}})) \in [N]^{|V(H)|}$.
    Since we have $\chi(Bor(d, \iota)[X]) > N^{|V(H)|}$, it follows that $c$ is not a proper coloring of $Bor(d, \iota)[X]$, so there exist $x, y \in X$ such that $xy \in E(Bor(d, \iota)[X])$ and $c(x) = c(y)$.
    We show that the map $\psi : V(H) \to [N], h \mapsto \phi(x_h)$ is a homomorphism from $H$ to $J$.
    Let $h, h' \in V(H)$ such that $hh' \in E(H)$.
    Then, $x_hy_{h'} \in E(H')$ so $\phi(x_h)\phi(y_{h'})$ is an edge of $J$.
    By definition, we have $\psi(h) = \phi(x_h)$ and $\psi(h') = \phi(x_{h'}) = \phi(y_{h'})$ since $c(x) = c(y)$.
    Thus, $\psi(h)\psi(h') \in E(J)$, so $\psi$ is indeed a homomorphism.
\end{proof}

\subsection{Proof of the lower bounds}

We prove the following result, which, together with \cref{lem:Myc-universal}, implies the lower bounds for \cref{thm:threshold-Kt,thm:threshold-odd-girth}.

\begin{proposition}\label{proof:lowerbound}
    For every $k \geq 2$, for every $t \geq 3$, we have \[\delta_{\hom}(\{C_3, C_5, \ldots, C_{2k+1}\}, \{C_3, \ldots, C_{2k-3}, \Gamma_{k-2}^t\}) \geq \frac{1}{3|V(\Gamma_{k-2}^t)|}.\]
\end{proposition}

\begin{proof}
    We show that for every $\varepsilon > 0$, and for every $N \geq 1$, there exists a graph $G$ which is $\{C_3, C_5, \ldots, C_{2k+1}\}$-free, has minimum degree at least $\left(\frac{1}{3|V(\Gamma_{k-2}^t)|} - \varepsilon\right)|V(G)|$ and has no homomorphism to a $\{C_3, \ldots,\allowbreak C_{2k-3},\allowbreak \Gamma_{k-2}^t\}$-free graph on at most $N$ vertices.
    Given $\varepsilon > 0$ and $N \geq 1$, let $G$ be the graph given by \cref{prop:construction} for $H = \Gamma_{k-2}^t$.
    Then, $og(G) \geq og(\Gamma_{k-2}^t) + 4 \geq 2k+3$ by \cref{lem:og-Myc}, so $G$ is indeed $\{C_3, C_5, \ldots, C_{2k+1}\}$-free and has minimum degree at least $\left(\frac{1}{3|V(\Gamma_{k-2}^t)|} - \varepsilon\right)|V(G)|$.
    By contradiction, suppose that $G$ has a homomorphism to a $\{C_3, \ldots, C_{2k-3}, \Gamma_{k-2}^t\}$-free graph $J$ on at most $N$ vertices.
    By construction of $G$, we have that $\Gamma_{k-2}^t$ has a homomorphism to $J$. Since $og(J) \geq 2k-1$, \cref{lem:Myc-core} implies that $J$ contains $\Gamma_{k-2}^t$ as a subgraph, a contradiction.
\end{proof}

\section{Existence of homomorphisms}

We first need some preliminary results.

\begin{lemma}\label{lem:almost-dom}
    Let $\delta, \varepsilon > 0$ and let $G$ be an $n$-vertex graph with minimum degree at least $\delta n$.
    There exists a set of at most $\lceil\log(1/\varepsilon)/\delta\rceil$ vertices of $G$ which dominates all but at most $\varepsilon n$ vertices of $G$.
\end{lemma}

\begin{proof}
    Since $G$ has minimum degree $\delta n$, a simple double-counting argument shows that for every set $V'$ of $k$ vertices of $G$, some vertex of $G$ has at least $\delta k$ neighbors in $V'$.
    Set $V_0 = V(G)$ and let $\ell = \lceil\log(1/\varepsilon)/\delta\rceil$.
    For every $i \in [\ell]$, let $v_i \in V(G)$ such that $v_i$ has at least $\delta |V_{i-1}|$ neighbors in $V_{i-1}$, and let $V_i = V_{i-1} \setminus N(v_i)$.
    Then, $\{v_1, \ldots, v_{\ell}\}$ dominates $V(G) \setminus V_{\ell}$ by definition.
    Furthermore, for every $i \in [\ell]$, we have $|V_i| \leq (1-\delta)|V_{i-1}|$, so $|V_{\ell}| \leq (1-\delta)^{\ell}n \leq e^{-\delta{\ell}}n \leq \varepsilon n$ by definition of $\ell$.
\end{proof}

We will use a slightly modified variant of the degree form of Szemerédi's Regularity Lemma \cite{S75}, see for instance \cite{KS96} for a closely related statement.
Recall that in a graph $G$, a pair $(U, W)$ of disjoint nonempty subsets of $V(G)$ is  \emph{$\varepsilon$-regular}, for $\varepsilon > 0$, if for every $U' \subseteq U$ and $W' \subseteq W$ such that $|U'| \geq \varepsilon|U|$ and $|W'| \geq \varepsilon|W|$, we have $|d(U, W) - d(U', W')| \leq \varepsilon$, where $d(U, W) = \frac{e(U, W)}{|U||W|}$.

\begin{lemma}\label{lem:deg-regularity}
    For every $\varepsilon > 0$ and $M_0\in \mathbb{N}$, there exists $M = M(\varepsilon,M_0)$ such that for any graph $G$, any partition $\mathcal{P}$ of $V(G)$ into at most $M_0$ pieces and any real number $\zeta \in [0, 1]$, there is a partition of $V(G)$ into sets $V_0, \ldots, V_p$ with $p\le M$ and a spanning subgraph $G'$ of $G$ with the following properties: \begin{itemize}
        \item
         $|V_0| \leq \varepsilon |V(G)|$,
        \item For every $i \geq 1$, $V_i$ is contained in a part of $\mathcal{P}$,
        \item $|V_i| = |V_j|$ for all $i, j \geq 1$,
        \item For every $v \in V(G)$, we have $d_{G'}(v) \geq d_G(v) - (\zeta + \varepsilon)|V(G)|$,
        \item $V_i$ is an independent set in $G'$ for all $i \geq 1$, 
        \item In $G'$, for all $i \neq j \in [p]$, the pair $(V_i, V_j)$ is $\varepsilon$-regular, with density either $0$ or greater than $\zeta$.
    \end{itemize}
\end{lemma}

We will also use the following basic fact on regular pairs.

\begin{lemma}\label{lem:regular-few-sparse}
    Let $G$ be a graph, let $\varepsilon > 0$, let $U, W$ be nonempty disjoint subsets of $V(G)$ such that the pair $(U, W)$ is $\varepsilon$-regular with density greater than $2\varepsilon$.
    Then, there are less than $\varepsilon|U|$ vertices $u \in U$ such that $|N_G(u) \cap W| < \varepsilon|W|$.
\end{lemma}

\begin{proof}
    Let $U' = \{u \in U : |N_G(u) \cap W| < \varepsilon|W|\}$.
    If $|U'| \geq \varepsilon|U|$, since the pair $(U, W)$ is regular, we would have $d(U', W) \geq \varepsilon$, which would imply that some vertex $u \in U'$ has at least $\varepsilon|W|$ neighbors in $W$, a contradiction.
\end{proof}

The following proposition yields the upper bounds of \cref{thm:threshold-Kt,thm:threshold-odd-girth}.

\begin{proposition}\label{prop:proof-upper-bound}
    For every $k \geq 2$, for every $t \geq 3$, we have \[\delta_{\hom}(\{C_3, C_5, \ldots, C_{2k+1}\}, \{C_3, \ldots, C_{2k-3}, \Gamma_{k-2}^t\}) \leq \frac{1}{3|V(\Gamma_{k-2}^t)|}.\]
\end{proposition}

Before presenting the proof of \cref{prop:proof-upper-bound} in detail, we give an overview of the idea in the special case of \cref{thm:threshold-triangle} (i.e.\ $k=2,t=3$). The proof can be seen as going in the opposite direction from the lower-bound construction. There, the graph is built from three pieces, each having minimum degree close to one third of its size, which leads to the threshold $1/9$. The upper bound follows by showing that any obstruction would give rise to an analogous three-part structure, which cannot fit once the relative minimum degree exceeds $1/9$.

We begin by finding a bounded set which dominates almost all vertices, and partition the dominated vertices so that any two vertices in the same part have a common neighbor. We then apply the Regularity Lemma, obtaining a bounded number of regularity classes and a cleaned subgraph $G'\subseteq G$ in which every pair of classes is either empty or regular of positive density, while decreasing each degree by only a small amount. For each vertex, we record approximately how many neighbors it has in each regularity class in $G'$, and group together vertices with the same profile. The odd girth assumption yields that every profile class is an independent set. Finally, we refine this partition once more so that short walks in the quotient can be lifted to short walks in $G$.

Suppose that the resulting quotient contains a triangle, with corresponding parts $Q_1,Q_2,Q_3$. For each $Q_j$, consider its \emph{support}, namely the set of regularity classes in which vertices of $Q_j$ have positive density in $G'$. The fact that $G$ contains no triangles or $5$-cycles implies that these three supports are pairwise disjoint, and moreover that there are no edges in $G'$ between any two classes belonging to their union. For each $j$, consider also the set of regularity classes adjacent in $G'$ to the support of $Q_j$. Again using the absence of short odd cycles, these three sets are pairwise disjoint. Since there are no edges between the supports, they are also disjoint from all three supports. Thus we obtain six pairwise disjoint sets of regularity classes.

The minimum degree condition, together with the odd girth assumption, implies that each support contains slightly more than $2/9$ of all regularity classes, while each of the three additional sets contains slightly more than $1/9$. Their total size is therefore greater than the total number of regularity classes, a contradiction.

This also explains the connection with the lower-bound construction: at the threshold $1/9$, three pieces of this type can fit together, while above $1/9$ the counting argument rules this out. The proof below carries out the same argument in the general setting of \cref{prop:proof-upper-bound}. The main additional difficulty is that, in the general case, the relevant obstruction is no longer a triangle, and the contradictions with the odd-girth condition may involve walks of length up to $2k-1$. We therefore iterate the refinement procedure sufficiently many times to ensure that all such short walks in the quotient can be lifted to $G$, inspired by similar arguments used in \cite{ES20}.
Roughly speaking, the role of the three vertices of the triangle is then played by the $s$ vertices of the corresponding Mycielskian: each gives a support of size slightly more than $2/(3s)$ and a disjoint neighborhood of size slightly more than $1/(3s)$, leading to the same counting contradiction. Again, this mirrors the structure of the lower-bound constructions, where each vertex of the obstruction gives rise to one such support-neighborhood pair.

\begin{proof}[Proof of \cref{prop:proof-upper-bound}.]
    Let $k \geq 2, t \geq 3$ and $\varepsilon > 0$.
    Let $s = |V(\Gamma_{k-2}^t)|$.
    Let $G$ be an $n$-vertex $\{C_3, C_5, \ldots, C_{2k+1}\}$-free graph with minimum degree at least $\left(\frac{1}{3s}+\varepsilon\right)n$.
    We prove that $G$ has a homomorphism to a $\{C_3, \ldots, C_{2k-3}, \Gamma_{k-2}^t\}$-free graph whose number of vertices only depends on $k, t, \varepsilon$.
    
    Let $\eta \coloneqq \varepsilon/7$.
    By \cref{lem:almost-dom}, there exists a set $D \subseteq V(G)$ of size at most $\gamma \coloneqq \lceil3s\log(1/\eta)\rceil$ which dominates all but at most $\eta \cdot n$ vertices of $G$.
    Let $H$ be the subgraph of $G$ induced by the vertices that are dominated by $D$ and let $\mathcal{P}$ be a partition of $V(H)$ into at most $\gamma$ parts such that any two vertices in a common part have a common neighbor: assign each dominated vertex to one chosen neighbor in $D$.
    Applying \cref{lem:deg-regularity} to $H$ with $\varepsilon = \eta$ and $\zeta = 2\eta$ (and adding $V(G) \setminus V(H)$ to $V_0$ and keeping all the edges incident to these vertices), there exists some $M=M(\eta,\gamma) = M(k, t, \varepsilon)$ and a partition of $V(G)$ into sets $V_0, \ldots, V_p$ and a spanning subgraph $G'$ of $G$ with the following properties: \begin{itemize}
        \item $p \leq M$, 
        \item $|V_0| \leq 2\eta \cdot n$,
        \item For every $i \geq 1$, any two vertices in $V_i$ have a common neighbor in $G$,
        \item There exists $m \in \mathbb{N}$ such that $|V_i| = m$ for every $i \geq 1$,
        \item For every $v \in V(G)$, we have $d_{G'}(v) \geq \left(\frac{1}{3s}+\varepsilon-3\eta\right)n$,
        \item $V_i$ is an independent set in $G'$ for all $i \geq 1$, 
        \item In $G'$, for all $i \neq j \in [p]$, the pair $(V_i, V_j)$ is $\eta$-regular, with density either $0$ or greater than $2\eta$.
    \end{itemize}

    For every $v \in V(G)$ and every $i \in [p]$, let $\mu_i(v) \in \{0, \eta, 2\eta, \ldots, \lfloor1/\eta\rfloor\eta\}$ such that $|N_{G'}(v) \cap V_i| \in [\mu_i(v) \cdot m, (\mu_i(v) + \eta) \cdot m]$.
    Informally, the vector $\mu(v)$ records the approximate density of $v$ on each of the regularity classes $V_1, \ldots, V_p$. 
    Let $\mathcal{Q}$ be the partition of the vertices of $G$ according to their vector $(\mu_1(v), \ldots, \mu_p(v))$.
    Note that $\mathcal{Q}$ need not refine the previous partitions.
    For every part $Q \in \mathcal{Q}$, all vertices $v \in Q$ have the same vector $(\mu_1(v), \ldots, \mu_p(v))$. We denote this vector by $\mu(Q) = (\mu_1(Q), \ldots, \mu_p(Q))$.
    The \emph{support} of $Q$ is the set $\{i \in [p] : \mu_i(Q) > 0\}$, which we denote by $\supp(Q)$.
    Observe that $\supp(Q)$ corresponds to a set of regularity classes.
    Note also that $M, \gamma$ and $\eta$ are independent of $n$, so $|\mathcal{Q}|$ only depends on $\varepsilon$, $k$ and $t$.

    \begin{claim}\label{cl:sum-mu}
        For every $Q \in \mathcal{Q}$, we have $\sum_{i \in [p]} \mu_i(Q) \geq \left(\frac{1}{3s}+\varepsilon-6\eta\right)p$.
    \end{claim}

    \begin{subproof}{\cref{cl:sum-mu}}
        Let $v \in Q$, so $\sum_{i \in [p]} \mu_i(Q) = \sum_{i \in [p]} \mu_i(v)$.
        For every $i \in [p]$, by definition of $\mu_i(v)$ we have $\mu_i(v) \cdot m \geq |N_{G'}(v) \cap V_i| - \eta \cdot m$.
        Thus, using that $mp \leq n$, we have \begin{align*}
            \sum_{i \in [p]} \mu_i(Q) &= \sum_{i \in [p]} \mu_i(v) \\
                &\geq \sum_{i \in [p]} \left(\frac{|N_{G'}(v) \cap V_i|}{m} - \eta\right) \\
                &= \left(\sum_{i \in [p]} \frac{|N_{G'}(v) \cap V_i|}{m}\right) - \eta p \\
                &= \frac{1}{m} \sum_{i \in \{0, \ldots, p\}} |N_{G'}(v) \cap V_i| - \frac{|N_{G'}(v) \cap V_0|}{m} - \eta p \\
                &\geq \frac{|N_{G'}(v)|}{m} - \frac{|V_0|}{m} - \eta p \\
                &\geq \left(\frac{1}{3s}+\varepsilon-3\eta\right) \cdot \frac{n}{m} - 2\eta\cdot\frac{n}{m} - \eta p \\
                &\geq \left(\frac{1}{3s} + \varepsilon - 6\eta\right)p. 
        \end{align*}
    \end{subproof}

    \begin{claim}\label{cl:Q-stable}
        Each $Q \in \mathcal{Q}$ is an independent set in $G$.
    \end{claim}

    \begin{subproof}{\cref{cl:Q-stable}}
        By contradiction, suppose that some $Q \in \mathcal{Q}$ is not an independent set in $G$, and let $x, y \in Q$ be adjacent in $G$.
        By \cref{cl:sum-mu}, there exists $i \in [p]$ such that $\mu_i(Q) > 0$, so $\mu_i(x), \mu_i(y) > 0$.
        Thus, both $x$ and $y$ have a neighbor in $V_i$.
        However, any two vertices in $V_i$ have a common neighbor in $G$, so $G$ contains a closed walk of length $5$, a contradiction since $G$ has odd girth greater than $2k+1 \geq 5$.
    \end{subproof}
    
    Let $\mathcal{R}^1 = \mathcal{Q}$ and define partitions $\mathcal{R}^2, \ldots, \mathcal{R}^k$ of $V(G)$ inductively as follows.
    Let $r \in \{2, \ldots, k\}$ and suppose that $\mathcal{R}^{r-1}$ was defined.
    Then, let $\mathcal{R}^{r}$ be the partition of $V(G)$ where two vertices are in the same part of $\mathcal{R}^r$ if and only if they are in the same part of $\mathcal{R}^{r-1}$ and their neighbors in $G$ belong to the same parts of $\mathcal{R}^{r-1}$.
    Note that $|\mathcal{R}^r| \leq |\mathcal{R}^{r-1}| \cdot 2^{|\mathcal{R}^{r-1}|}$, so $|\mathcal{R}^k|$ only depends on $\varepsilon$, $k$ and $t$.

    Let $G/\mathcal{R}^k$ be the graph on vertex set $\mathcal{R}^k$ where $R^k_i, R^k_j \in \mathcal{R}^k$ are adjacent if and only if there is an edge between $R^k_i$ and $R^k_j$ in $G$.
    Since $\mathcal{R}^k$ refines $\mathcal{Q}$, there is a natural homomorphism from $G$ to $G/\mathcal{R}^k$ by \cref{cl:Q-stable}, and the size of $G/\mathcal{R}^k$ is upper bounded by a function of $\varepsilon$, $t$ and $k$ only.
    We now show that $G/\mathcal{R}^k$ is $\{C_3, \ldots, C_{2k-3}, \Gamma_{k-2}^t\}$-free, which will conclude the proof. 
    By contradiction, suppose not.
    Then, there exist $R^k_1, \ldots, R^k_{s} \in \mathcal{R}^k$ with the following properties: \begin{itemize}
        \item For all $i \in [s]$, there exists a closed walk of length $2k-1$ in $G/\mathcal{R}^k$ that visits $R^k_i$,
        \item For all $i \neq j \in [s]$, there exist an even walk and an odd walk between $R^k_i$ and $R^k_j$ in $G/\mathcal{R}^k$, both of length at most $2k-2$.
    \end{itemize}
    Indeed, if $G/\mathcal{R}^k$ contains an odd cycle of length at most $2k-3$, we can take all $R^k_i$ equal to some part of $\mathcal{R}^k$ that belongs to such a cycle, and if $G/\mathcal{R}^k$ contains $\Gamma_{k-2}^t$, we can take the $R^k_i$ as the vertices of a copy of $\Gamma_{k-2}^t$ in $G/\mathcal{R}^k$. In the latter case, the two properties follow immediately from \cref{lem:Myc-2-vtces-odd-cycle}.

    For every $r \in [k]$ and every $i \in [s]$, let $R^r_i$ be the part of $\mathcal{R}^r$ that contains $R^k_i$.
    For convenience, let $Q_i = R^1_i$ for every $i \in [s]$.

    \begin{claim}\label{cl:exists-walk-back}
        For every $i \in [s]$, there exists a walk of length $2k-1$ in $G$ between two vertices of $Q_i$.
    \end{claim}

    \begin{subproof}{\cref{cl:exists-walk-back}}
        By symmetry, it suffices to prove it for $i = 1$.
        Let $R^k_1 = S^k_1S^k_2\ldots S^k_{2k-1}S^k_{2k} = R^k_1$ be a closed walk in $G/\mathcal{R}^k$, where $S^k_i \in \mathcal{R}^k$ for every $i \in [2k]$.
        For every $r \in [k]$ and every $i \in [2k]$, let $S^r_i$ be the part of $\mathcal{R}^r$ that contains $S^k_i$.
        Since $S^k_{k}$ and $S^k_{k+1}$ are adjacent in $G/\mathcal{R}^k$, there exist $v_{k} \in S^k_{k}$ and $v_{k+1} \in S^k_{k+1}$ that are adjacent in $G$.
        Furthermore, there is an edge in $G$ between $S^k_{k}$ and $S^k_{k-1}$ and between $S^k_{k+1}$ and $S^k_{k+2}$.
        By definition of $\mathcal{R}^k$, this implies that $v_{k}$ is adjacent in $G$ to some vertex $v_{k-1} \in S^{k-1}_{k-1}$ and $v_{k+1}$ is adjacent in $G$ to some vertex $v_{k+2} \in S^{k-1}_{k+2}$.
        After repeating this process $k-2$ extra times, we finally obtain a vertex $v_1 \in R^{1}_1$ and a vertex $v_{2k} \in R^{1}_1$.
        Then, $v_1, v_{2k} \in Q_1$ and they are connected by a walk of length $2k-1$ in $G$.
    \end{subproof}

    A similar proof shows that the following claim also holds.

    \begin{claim}\label{cl:exists-walk-between}
        For all $i \neq j \in [s]$, there exist an even walk of length at most $2k-2$ in $G$ between a vertex of $Q_i$ and a vertex of $Q_j$, and an odd walk of length at most $2k-2$ in $G$ between a vertex of $Q_i$ and a vertex of $Q_j$.
    \end{claim}

    \begin{claim}\label{cl:disjoint-supports}
        The sets $\supp(Q_1), \ldots, \supp(Q_{s})$ are pairwise disjoint.
    \end{claim}

    \begin{subproof}{\cref{cl:disjoint-supports}}
        By contradiction, suppose not: there exist $j \neq j' \in [s]$ and $i \in \supp(Q_j) \cap \supp(Q_{j'})$.
        By \cref{cl:exists-walk-between}, there exist $x_j \in Q_j$ and $x_{j'} \in Q_{j'}$ that are connected by a walk in $G$ with an odd number of edges, and at most $2k-2$ edges (hence at most $2k-3$ edges).
        Since $i \in \supp(Q_j) \cap \supp(Q_{j'})$, we have $\mu_i(x_j), \mu_i(x_{j'}) > 0$ so both $x_j$ and $x_{j'}$ have a neighbor in the regularity class $V_i$.
        However, any two vertices in $V_i$ have a common neighbor, so $x_j$ and $x_{j'}$ are connected by a walk in $G$ with $4$ edges.
        This proves the existence of a closed odd walk of length at most $2k+1$ in $G$, a contradiction.
    \end{subproof}

    \begin{claim}\label{cl:large-supports}
        For every $j \in [s]$, we have $|\supp(Q_j)| \geq 2\left(\frac{1}{3s}+\varepsilon-6\eta\right)p$.
    \end{claim}

    \begin{subproof}{\cref{cl:large-supports}}
        Let $j \in [s]$ and let $x_j, y_j \in Q_j$ be connected by a walk of length $2k-1$ in $G$; such vertices exist by \cref{cl:exists-walk-back}.
        We first show that $\mu_i(Q_j) \leq 1/2$ for every $i \in [p]$.
        By contradiction, suppose not. 
        Then, $\mu_i(x_j), \mu_i(y_j) > 1/2$ for some $i \in [p]$.
        Thus, $|N_{G'}(x_j) \cap V_i|, |N_{G'}(y_j) \cap V_i| > |V_i|/2$ so $x_j$ and $y_j$ have a common neighbor. Together with the walk of length $2k-1$ between them, this forms a closed walk of length $2k+1$ in $G$, so $G$ contains a closed odd walk of length at most $2k+1$, a contradiction.
        This proves that $\mu_i(Q_j) \leq 1/2$ for every $i \in [p]$.
        Then, using \cref{cl:sum-mu}, we have \begin{align*}
            |\supp(Q_j)| &\geq \sum_{i \in \supp(Q_j)}2\mu_i(Q_j) \\
            &= 2 \sum_{i \in \supp(Q_j)}\mu_i(Q_j) \\
            &\geq 2\left(\frac{1}{3s}+\varepsilon-6\eta\right)p,
        \end{align*}
        as claimed.
    \end{subproof}

    \begin{claim}\label{cl:supp-edgeless}
        If $i, i' \in \bigcup_{j \in [s]} \supp(Q_j)$, there is no edge between $V_i$ and $V_{i'}$ in $G'$.
    \end{claim}

    \begin{subproof}{\cref{cl:supp-edgeless}}
        For $i = i'$, this is immediate since $V_i$ is an independent set in $G'$.
        Suppose now that $i \neq i'$, and suppose by contradiction that there is an edge between $V_i$ and $V_{i'}$ in $G'$.
        Thus, the pair $(V_i, V_{i'})$ is $\eta$-regular in $G'$, with density greater than $2\eta$.
        Consider first the case where $i, i'$ are in the support of the same $Q_j$, and let $y \in Q_j$.
        Since $i, i' \in \supp(Q_j)$, we have $\mu_i(y), \mu_{i'}(y) > 0$ so by definition $y$ has at least $\eta \cdot m$ neighbors both in $V_i$ and in $V_{i'}$.
        Since $(V_i, V_{i'})$ is $\eta$-regular in $G'$, with density greater than $2\eta$, there is an edge in $G'$ between $N_{G'}(y) \cap V_i$ and $N_{G'}(y) \cap V_{i'}$. 
        This gives a triangle in $G'$, which is impossible.
        Consider now the case where $i \in \supp(Q_j)$ and $i' \in \supp(Q_{j'})$ with $j \neq j'$.
        By \cref{cl:exists-walk-between}, there exist $x_j \in Q_j$ and $x_{j'} \in Q_{j'}$ that are connected by a walk in $G$ with an even number of edges, and at most $2k-2$ edges.
        By definition of $\mu$, $x_j$ has at least $\eta \cdot m$ neighbors in $V_i$ and $x_{j'}$ has at least $\eta \cdot m$ neighbors in $V_{i'}$.
        Since $(V_i, V_{i'})$ is $\eta$-regular in $G'$, with density greater than $2\eta$, there is an edge in $G'$ between $N_{G'}(x_j) \cap V_i$ and $N_{G'}(x_{j'}) \cap V_{i'}$.
        Thus, $x_j$ and $x_{j'}$ are connected by a walk in $G$ with $3$ edges.
        This proves the existence of a closed odd walk of length at most $2k+1$ in $G$, which is impossible.
    \end{subproof}

    For every $Q \in \mathcal{Q}$, let $\supp^2(Q)$ be the set of indices $i \in [p]$ such that there is an edge in $G'$ between $V_i$ and $V_{i'}$ for some $i' \in \supp(Q)$.
    Again, note that $\supp^2(Q)$ corresponds to a set of regularity classes.

    \begin{claim}\label{cl:supp2-disjoint}
        The sets $\supp^2(Q_1), \ldots, \supp^2(Q_{s})$ are pairwise disjoint.
    \end{claim}

    \begin{subproof}{\cref{cl:supp2-disjoint}}
        Let $j \neq j' \in [s]$ and suppose by contradiction that there exists $\ell \in \supp^2(Q_j) \cap \supp^2(Q_{j'})$.
        Then, there exist $i \in \supp(Q_j)$ and $i' \in \supp(Q_{j'})$ such that there is an edge in $G'$ between $V_{\ell}$ and $V_i$, and between $V_{\ell}$ and $V_{i'}$.
        By \cref{cl:exists-walk-between}, there exist $x_j \in Q_j$ and $x_{j'} \in Q_{j'}$ that are connected by a walk in $G$ with an odd number of edges, and at most $2k-2$ edges (hence at most $2k-3$ edges).
        Since $i \in \supp(Q_j)$, we have $|N_{G'}(x_j) \cap V_{i}| \geq \eta \cdot m$.
        Similarly, we have $|N_{G'}(x_{j'}) \cap V_{i'}| \geq \eta \cdot m$.
        The pair $(V_i, V_{\ell})$ is $\eta$-regular with density greater than $2\eta$ in $G'$ so by \cref{lem:regular-few-sparse}, the number of vertices $u \in V_i$ such that $|N_{G'}(u) \cap V_{\ell}| < \eta|V_{\ell}|$ is less than $\eta|V_i|$.
        Thus, there exists $z \in N_{G'}(x_j) \cap V_i$ such that $|N_{G'}(z) \cap V_{\ell}| \geq \eta|V_{\ell}|$.
        Since $(V_{\ell}, V_{i'})$ is $\eta$-regular in $G'$, with density greater than $2\eta$, there is an edge in $G'$ between $N_{G'}(z) \cap V_{\ell}$ and $N_{G'}(x_{j'}) \cap V_{i'}$. 
        Thus, there is a walk of length $4$ between $x_j$ and $x_{j'}$ in $G$.
        This proves the existence of a closed odd walk of length at most $2k+1$ in $G$, which is impossible.
    \end{subproof}

    \begin{claim}\label{cl:size-supp2}
        For every $j \in [s]$, we have $|\supp^2(Q_j)| \geq \left(\frac{1}{3s}+\varepsilon-6\eta\right)p$.
    \end{claim}

    \begin{subproof}
        Let $j \in [s]$, let $i \in \supp(Q_j)$ and let $v \in V_i$.
        Let $Q \in \mathcal{Q}$ such that $v \in Q$.
        It follows from \cref{cl:sum-mu} that $|\supp(Q)| \geq \left(\frac{1}{3s}+\varepsilon-6\eta\right)p$.
        However, if $i' \in \supp(Q)$ then $\mu_{i'}(v) > 0$ so there is an edge between $v \in V_i$ and $V_{i'}$ in $G'$.
        Thus, $i' \in \supp^2(Q_j)$.
        Therefore, $|\supp^2(Q_j)| \geq \left(\frac{1}{3s}+\varepsilon-6\eta\right)p$.
    \end{subproof}

    \begin{claim}\label{cl:all-disjoint}
        The sets $\supp(Q_1), \ldots, \supp(Q_{s}), \supp^2(Q_1), \ldots, \supp^2(Q_{s})$ are pairwise disjoint.
    \end{claim}

    \begin{subproof}{\cref{cl:all-disjoint}}
        By \cref{cl:disjoint-supports} and \cref{cl:supp2-disjoint}, it suffices to show that $\supp(Q_j)$ and $\supp^2(Q_{j'})$ are disjoint for all $j, j' \in [s]$.
        Let $i' \in \supp^2(Q_1) \cup \ldots \cup \supp^2(Q_{s})$. 
        There exists $i \in \supp(Q_1) \cup \ldots \cup \supp(Q_{s})$ such that there is an edge between $V_i$ and $V_{i'}$ in $G'$ (hence in $G$), so $i' \notin \supp(Q_1) \cup \ldots \cup \supp(Q_{s})$ by \cref{cl:supp-edgeless}.
    \end{subproof}
    
    By \cref{cl:all-disjoint,cl:large-supports,cl:size-supp2}, we then have \[\left|\bigcup_{j \in [s]} \supp(Q_j) \cup \supp^2(Q_j)\right| = \sum_{j \in [s]} |\supp(Q_j)| + |\supp^2(Q_j)| \geq s \cdot 3 \cdot \left(\frac{1}{3s}+\varepsilon-6\eta\right)p > p,\] a contradiction. 
\end{proof}

For completeness, we briefly deduce the statements of \cref{thm:threshold-Kt,thm:threshold-odd-girth,thm:threshold-triangle} from the preceding results.

\begin{proof}[Proof of~\cref{thm:threshold-Kt,thm:threshold-odd-girth,thm:threshold-triangle}]
By \cref{proof:lowerbound,prop:proof-upper-bound} we have that \[\delta_{\hom}(\{C_3, C_5, \ldots, C_{2k+1}\}, \{C_3, \ldots, C_{2k-3}, \Gamma_{k-2}^t\}) = \frac{1}{3|V(\Gamma_{k-2}^t)|}\] for every $k \geq 2$ and $t \geq 3$. 

Observe that by the definition of the generalized Mycielskians, it follows that $\Gamma_0^t=K_t$ for every $t\ge 2$. Hence, plugging $k=2$ into the above we immediately find that
\[\delta_{\hom}(\{C_3, C_5\}, \{K_t\}) =\frac{1}{3t}\] for every $t \geq 3$. This establishes the statement of~\cref{thm:threshold-Kt}, and hence also \cref{thm:threshold-triangle}.

Now, for~\cref{thm:threshold-odd-girth}, let $k\ge 2$ and a graph $H$ of odd girth at least $2k-1$ (and hence greater than $2(k-2)+1$) be given. Then by applying~\cref{lem:Myc-universal} to the graph $H$ and with parameter $k-2\ge 0$ in place of $k$, we find that there exists some $t\ge 3$ such that $H$ is a subgraph of $\Gamma_{k-2}^t$. Hence, any bounded-size $\{C_3, \ldots, C_{2k-3}, H\}$-free graph is automatically also a bounded-size $\{C_3,\ldots,C_{2k-3},\Gamma_{k-2}^t\}$-free graph. It follows that 
\begin{align*}\delta_{\hom}(\{C_3,C_5,\ldots,C_{2k+1}\},\{C_3,\ldots,C_{2k-3},H\})&\ge \delta_{\hom}(\{C_3,C_5,\ldots,C_{2k+1}\},\{C_3,\ldots,C_{2k-3},\Gamma_{k-2}^t\})\\ &=\frac{1}{3|V(\Gamma_{k-2}^t)|}\\&>0.\end{align*} This establishes also the statement of~\cref{thm:threshold-odd-girth}, concluding the proof.
\end{proof}

\bibliographystyle{alphaurl}
\bibliography{biblio}

\end{document}